\documentclass[a4paper,11pt]{article}

\usepackage{amsfonts,amsthm,amssymb,amsmath}
\usepackage{hyperref}
\usepackage{fourier}
\usepackage{color}

\newtheorem{theorem}{Theorem}[section]

\theoremstyle{definition}

\newtheorem{remark}[theorem]{Remark}

\DeclareMathOperator{\card}{card}

\usepackage{graphicx}
\usepackage{fancybox}

\DeclareMathOperator{\var}{vars}

\title{The Bohnenblust--Hille inequality combined with an inequality of Helson}

\author{Daniel Carando\footnote{Departamento de Matematica - Pab I,
Facultad de Cs. Exactas y Naturales, Universidad de Buenos Aires, 1428 Buenos Aires (Argentina) and IMAS - CONICET (Argentina). Partially supported by CONICET-PIP 0624, PICT 2011-1456 and UBACyT Grant 1-746}
\and
Andreas Defant\footnote{Institut f\"ur Mathematik. Universit\"at Oldenburg. D-26111 Oldenburg (Germany). Supported by MICINN  MTM2011-22417 }
\and
Pablo Sevilla-Peris \footnote{Instituto Universitario de Matem\'atica Pura y Aplicada. Universitat Polit\`ecnica de Val\`encia. 46022 Valencia (Spain). Supported by MICINN  MTM2011-22417 and UPV-SP201207000 }}

\date{}

\begin{document}

\maketitle

\begin{abstract}
\noindent We give a variant of the Bohenblust-Hille inequality which, for certain families of polynomials, leads to constants with polynomial growth in the degree.
\end{abstract}

\section{Introduction}
Hardy and Littlewood showed in \cite{HaLi32} that there exists a constant $K>0$ such that for every $f\in H^{1}$ we have
\begin{equation*}
 \bigg( \int_{\mathbb{D}}  \vert f(z) \vert^{2} dm(z) \bigg)^{1/2} \leq K \int_{\mathbb{T}} \vert f(w) \vert d\sigma(w) \,,
\end{equation*}
where $dm$ and $d\sigma$ denote respectively the normalised Lebesgue measures on the complex unit disk $\mathbb{D}$ and the torus (or unit circle) $\mathbb{T}$.  Equivalently, this means that  the Hardy space $H_1(\mathbb{T})$ is contained in the Bergman
space $B_2(\mathbb{D})$. Shapiro \cite[p.~117-118]{Sh77} showed that the inequality holds with $K=\pi$ and
Mateljevi\'{c} \cite{Ma79} (see also \cite{Ma80,Vu03})
showed that  actually the constant could be taken $K=1$. A simple
reformulation of the Bergman norm then gives that  if $\sum_{n=0}^{\infty} a_{n} z^{n}$ is the Fourier series expansion of $f\in H^{1}(\mathbb{D})$ we have
\[
 \bigg( \sum_{n=0}^{\infty} \frac{\vert a_{n} \vert^{2}}{n+1} \bigg)^{1/2} \leq \int_{\mathbb{T}} \vert f(w) \vert d\sigma(w) \,.
\]
A few years later Helson in \cite{He06}  generalised this inequality to functions in $N$ variables. For $n \in \mathbb{N}$  denote by $d(n)$  the number of divisors  and by $p^{\alpha}= p_{1}^{\alpha_{1}} \cdots p_{k}^{\alpha_{k}}$  the prime decomposition of $n$. Then
we have that for every $f \in H^{1}(\mathbb{T}^{N})$ with Fourier series expansion $\sum_{\alpha \in \mathbb{N}_{0}^{N}} c_{\alpha} z^{\alpha}$
\begin{equation} \label{Helson}
  \bigg( \sum_{\alpha \in \mathbb{N}_{0}^{N}} \frac{\vert c_{\alpha} \vert^{2}}{d(p^{\alpha})} \bigg)^{1/2} \leq \int_{\mathbb{T}^{N}} \vert f(w) \vert d\sigma(w) \,.
\end{equation}
Given a multiindex $\alpha$, we write $\alpha + 1=(\alpha_{1}+1) \cdots (\alpha_{k}+1)$.
Note that, with this notation, we have $d(p^{\alpha}) = \alpha + 1$.\\

On the other hand, by the Bohnenblust-Hille inequality \cite{BoHi31} as presented in \cite{DeFrOrOuSe11} there is a constant $C>0$ such that for every $m$-homogeneous polynomial in $N$ variables $P (z) = \sum_{\vert \alpha \vert = m} c_{\alpha} z^{\alpha}$ with
$z \in \mathbb{C}^{N}$ we have
\begin{equation} \label{Bohnenblust Hille}
 \bigg( \sum_{\vert \alpha \vert = m} \vert c_{\alpha} \vert^{\frac{2m}{m+1}} \bigg)^{\frac{m+1}{2m}} \leq C^{m} \sup_{z \in \mathbb{D}^{N}} \vert P(z) \vert \, .
\end{equation}
The proof of this inequality given in \cite{DeFrOrOuSe11} consists basically of two steps: first to decompose  the sum in \eqref{Bohnenblust Hille} as the product of certain mixed sums and second to bound each one of these sums by a term
including $\Vert P \Vert$, the supremum of $|P|$ in $\mathbb{D}^{N}$. For this second step usually the following result of Bayart \cite{Ba02} is used: for every $m$-homogeneous polynomial in $N$ variables we have
\begin{equation} \label{Bayart}
 \bigg( \sum_{\vert \alpha \vert = m} \vert c_{\alpha} \vert^{2}  \bigg)^{1/2} \leq 2^{m/2} \int_{\mathbb{T}^{N}} \Big\vert \sum_{\vert \alpha \vert = m} c_{\alpha} w^{\alpha} \Big\vert d\sigma (w) \,.
\end{equation}
Very recently, it was proved in  \cite[Corollary 5.3]{BaPeSe13} that for every $\varepsilon>0$ there exists $\kappa>0$ such that we can take $\kappa (1+\varepsilon)^m$ as the constant in \eqref{Bohnenblust Hille}. Our aim in this note is get a variant
of  \eqref{Bohnenblust Hille} by using  \eqref{Helson} instead of \eqref{Bayart}. With this variant, we see that for polynomials $P$ each of whose monomials involve a uniformly bounded number of variables, the obtained constants have polynomial growth in $m$.

\section{Main result and some remarks}

The following is our main result.

\begin{theorem}\label{Komische BH}
 Let $\Lambda \subseteq \{ \alpha \in \mathbb{N}_{0}^{N} \colon \vert \alpha \vert =m \}$ be an indexing set. Then for every family $\big( c_{\alpha} \big)_{\alpha \in \Lambda}$ we have
\[
\bigg( \sum_{\alpha \in \Lambda} \Big( \frac{\vert c_{\alpha} \vert }{\sqrt{\alpha + 1}} \Big)^{\frac{2m}{m+1}} \bigg)^{\frac{m+1}{2m}}
\leq m^{\frac{m-1}{2m}} \Big( 1 - \frac{1}{m-1} \Big)^{m-1} \sup_{z \in \mathbb{D}^{N}} \Big\vert \sum_{\alpha \in \Lambda}c_{\alpha} z^{\alpha} \Big\vert \, .
\]
\end{theorem}
\noindent We give  several remarks before  we present the proof.

\begin{remark} \text{}
\begin{enumerate}
\item
It is easy to see that $\sqrt{\alpha + 1} \leq \sqrt{2}^m$. Hence the preceding inequality includes the hypercontractive version of the Bohnenblust-Hille inequality from \eqref{Bohnenblust Hille} as a special case.

\item
Thanks to the term $\sqrt{\alpha+1}$, the constants in the previous inequality  grow much more slowly than the constants in \eqref{Bohnenblust Hille}. Actually, we have
$$m^{\frac{m-1}{2m}} \big( 1 - \frac{1}{m-1} \big)^{m-1} = \frac{\sqrt{m}} e\ \big( 1  + o(m) \big).$$
\item
Let $\var(\alpha)$ denote the numbers of different variables involved in the monomial $z^\alpha$. In other words, $\var(\alpha)=\card\big\{j:\alpha_j\ne 0\big\}$.
Given $M$ we consider the set
\[
\Lambda_{N,M} = \big\{\alpha \in \mathbb{N}_{0}^{N} \colon \vert \alpha \vert =m \text{ and } \var(\alpha)\le M \big\}\,,
\]
(note that if $M\geq N$, then $\Lambda_{N,M}=\Lambda_{N,N}$).
An application of Lagrange multipliers gives that for any $\alpha\in \Lambda_{N,M}$ we have for every $N$ and $M$
$$\alpha+1= (\alpha_{1}+1) \cdots (\alpha_{k}+1) \cdots \le \Big( \frac m M + 1 \Big)^M.$$
Combining this with Theorem~\ref{Komische BH} we obtain for every $m$, $N$, $M$
\begin{equation*}
\begin{split}
\bigg( \sum_{\alpha \in \Lambda_{N,M}}  {\vert c_{\alpha} \vert } ^{\frac{2m}{m+1}} \bigg)^{\frac{m+1}{2m}}
&
\le
\Big( \frac m M + 1 \Big)^{M/2} \bigg( \sum_{\alpha \in \Lambda_{N,M}} \Big( \frac{\vert c_{\alpha} \vert }{\sqrt{\alpha + 1}} \Big)^{\frac{2m}{m+1}} \bigg)^{\frac{m+1}{2m}}
\\&
\le  \Big( \frac m M + 1 \Big)^{M/2}  m^{\frac{m-1}{2m}} \Big( 1 - \frac{1}{m-1} \Big)^{m-1} \sup_{z \in \mathbb{D}^{N}} \Big\vert \sum_{\alpha \in \Lambda_{N,M}}c_{\alpha} z^{\alpha} \Big\vert \, ,
\end{split}
\end{equation*}
hence
\begin{equation} \label{Nvariables}
\begin{split}
\bigg( \sum_{\alpha \in \Lambda_{N,M}}  {\vert c_{\alpha} \vert } ^{\frac{2m}{m+1}} \bigg)^{\frac{m+1}{2m}} \leq 2^{\frac{M}{2}}  m^{\frac {M+1} 2} \ \sup_{z \in \mathbb{D}^{N}} \Big\vert \sum_{\alpha \in \Lambda_{N,M}}c_{\alpha} z^{\alpha} \Big\vert \,.
\end{split}
\end{equation}
This means that for polynomials whose monomials have a uniformly boun\-ded number $M$ of different variables,  we get a Bohnenblust-Hille type inequality with a constant of polynomial growth in $m$.   We remark
that the dimension $N$ plays no role in this inequality, the only important point here is the number of different variables in each monomial. As a consequence, an analogue of \eqref{Nvariables} holds for  $m$-homogeneous  polynomials on $c_0$:
 Let $P:c_0\to \mathbb C$ be  an $m$-homogeneous  polynomial and
\[
 \Lambda_{M} = \{\alpha \in \mathbb{\mathbb N}_{0}^{(\mathbb N)} \colon \vert \alpha \vert =m \text{ and } \var(\alpha)\le M\} \,.
\]
Then  for every $M$ and $m$
\[
\bigg( \sum_{\alpha \in \Lambda_{M}}  {\vert c_{\alpha}(P) \vert } ^{\frac{2m}{m+1}} \bigg)^{\frac{m+1}{2m}}
\leq 2^{\frac{M}{2}}  m^{\frac {M+1} 2} \ \|P\| \, ,
\]
where the $c_\alpha(P)$ are the coefficients of $P$ and $\|P\|$ is the supremum of $|P|$ on the unit ball of $c_0$.

\item In  \cite[Theorem~5.3]{DeMaSc12} a very general version of the Bohnenblust-Hille inequality is given, involving operators with values on a Banach lattice. A straightforward combination of the proof of Theorem~\ref{Komische BH} (see the final
section) and the arguments presented in \cite[Theorem~5.3]{DeMaSc12} easily gives a version of Theorem~\ref{Komische BH} in that setting.

\end{enumerate}
\end{remark}

\medskip

\section{The proof}

Let us fix some notation before we prove our main result. We are going to use the following indexing sets
\begin{gather*}
  \mathcal{M}(m,N) = \{ \boldsymbol{i}= (i_{1}, \ldots , i_{m} ) \colon 1 \leq i_{j} \leq N , \, j=1, \ldots , m \} \\
  \mathcal{J}(m,N) = \{ \boldsymbol{i} \in  \mathcal{M}(m,N)  \in  \colon 1 \leq i_{1} \leq \cdots \leq i_{m} \leq N  \} \, .
\end{gather*}
In $\mathcal{M}(m,N)$ we define an equivalence relation by $\boldsymbol{i} \sim \boldsymbol{j}$ if there is a permutation $\sigma$ of $\{1, \ldots , N\}$ such that $j_{k}=i_{\sigma(k)}$ for
every $k$. With this, if $\big( a_{i_{1}, \ldots , i_{m}} \big)$ are symmetric then we have
\[
  \sum_{\boldsymbol{i} \in \mathcal{M}(m,N)} a_{\boldsymbol{i} } = \sum_{\boldsymbol{i} \in   \mathcal{J}(m,N)} \sum_{\boldsymbol{j} \in [\boldsymbol{i}]}  a_{\boldsymbol{j} }
=  \sum_{\boldsymbol{i} \in   \mathcal{J}(m,N)} \card [\boldsymbol{i}]  a_{\boldsymbol{i} } \, .
\]
Also, given $\boldsymbol{i} \in \mathcal{M}(m-1,N)$ and $j \in \{1, \ldots , N\}$, for $1 \leq k \leq m-1$ we define $(\boldsymbol{i},_{k} j) = (i_{1}, \ldots , i_{k-1},j,i_{k} , \ldots , i_{m-1} )
\in  \mathcal{M}(m,N)$ (that is, we put $j$ in the $k$-th possition, shifting the rest to the right).\\
There is a one-to-one correspondance between $ \mathcal{J}(m,N)$ and $\{ \alpha \in \mathbb{N}_{0}^{N} \colon \vert \alpha \vert =m \}$ defined as follows. For each $\boldsymbol{i}$
we define $\alpha = (\alpha_{1}, \ldots , \alpha_{N})$ by $\alpha_{r} = \card \{ j \colon i_{j}=r \}$ (i.e. $\alpha_{r}$ counts how many times $r$ comes in $\boldsymbol{i}$); on the other hand,
given $\alpha$ we define $\boldsymbol{i} = (1, \stackrel{\alpha_{1} }{\ldots} , 1 , \ldots , N , \stackrel{\alpha_{N} }{\ldots},N) \in  \mathcal{J}(m,N)$.\\
Each $m$-homogeneous polynomial on $N$ variables has a unique symmetric  $m$-linear form $L: \mathbb{C}^{N} \times \cdots \times \mathbb{C}^{N} \to \mathbb{C}$ such that $P(z) =
L(z, \ldots , z)$ for every $z$. If $(c_{\alpha})$ are the coefficients of the polynomial and $a_{i_{1}, \ldots, i_{m}} = L(e_{i_{1} }, \ldots , e_{i_{m}})$ is the matrix of $L$ we have
$ c_{\alpha} = \card[\boldsymbol{i}] a_{\boldsymbol{i}}$, where $\alpha$ and $\boldsymbol{i}$ are related to each other.\\
Finally, if $\alpha$ and $\boldsymbol{i}$ are related and $p_{1} < p_{2} < \cdots$ denotes the sequence of prime numbers, we will write $p^{\alpha} = p_{1}^{\alpha_{1}} \cdots p_{N}^{\alpha_{N}}
= p_{i_{1}} \cdots p_{i_{m}} = p_{\boldsymbol{i}}$.

\begin{proof}[Proof of Theorem~\ref{Komische BH}]
We follow essentially the guidelines of the proof of the Bohnenblust-Hille inequality as presented in \cite{DeFrOrOuSe11}. First of all let us assume that $c_{\alpha}=0$ for every $\alpha \not\in \Lambda$; then we have
\begin{multline*}
 \bigg( \sum_{\alpha \in \Lambda} \Big( \frac{\vert c_{\alpha} \vert }{\sqrt{\alpha + 1}} \Big)^{\frac{2m}{m+1}} \bigg)^{\frac{m+1}{2m}}
=  \bigg( \sum_{\boldsymbol{i} \in \mathcal{J}(m,N)} \Big\vert \card[\boldsymbol{i}] \frac{a_{\boldsymbol{i}}}{\sqrt{d(p_{\boldsymbol{i}} ) }} \Big\vert^{\frac{2m}{m+1}}  \bigg)^{\frac{m+1}{2m}} \\
=  \bigg( \sum_{\boldsymbol{i} \in \mathcal{M}(m,N)} \frac{1}{\card[\boldsymbol{i}]}\Big\vert \card[\boldsymbol{i}] \frac{a_{\boldsymbol{i}}}{\sqrt{d(p_{\boldsymbol{i}} ) }} \Big\vert^{\frac{2m}{m+1}}  \bigg)^{\frac{m+1}{2m}} \\
=  \bigg( \sum_{\boldsymbol{i} \in \mathcal{M}(m,N)} \Big\vert \card[\boldsymbol{i}]^{1-\frac{m+1}{2m}} \frac{a_{\boldsymbol{i}}}{\sqrt{d(p_{\boldsymbol{i}} ) }} \Big\vert^{\frac{2m}{m+1}}  \bigg)^{\frac{m+1}{2m}} \, .
\end{multline*}
We now use an inequality due to Blei \cite[Lemma~5.3]{Bl79} (see also \cite[Lemma~1]{DeFrOrOuSe11}): for any family of complex numbers $\big(b_{\boldsymbol{i} } \big)_{\boldsymbol{i} \in  \mathcal{M}(m,N)}$ we have
\begin{equation} \label{Blei}
  \sum_{\boldsymbol{i} \in  \mathcal{M}(m,N)} \vert b_{\boldsymbol{i} } \vert^{\frac{2m}{m+1}}
\leq \prod_{k=1}^{m} \bigg( \sum_{j=1}^{N} \Big( \sum_{\boldsymbol{i} \in \mathcal{M}(m-1,N) } \vert b_{(\boldsymbol{i},_{k} j )} \vert^{2} \Big)^{1/2} \bigg)^{\frac{2}{m-1}} \, .
\end{equation}
Using this and the fact that $\card[(\boldsymbol{i},_{k}j)] \leq m \card[\boldsymbol{i}]$ we get
\begin{align*}
  \bigg( \sum_{\alpha \in \Lambda} \Big( \frac{\vert c_{\alpha} \vert }{\sqrt{\alpha + 1}} & \Big)^{\frac{2m}{m+1}} \bigg)^{\frac{m+1}{2m}} \\
& \leq \prod_{k=1}^{m} \bigg( \sum_{j=1}^{N} \Big( \sum_{\boldsymbol{i} \in \mathcal{M}(m-1,N) } \Big\vert \card[(\boldsymbol{i},_{k} j )]^{\frac{m-1}{2m}} \frac{a_{(\boldsymbol{i},_{k} j )}}{\sqrt{d(p_{(\boldsymbol{i},_{k} j )} ) }} \Big\vert^{2} \Big)^{1/2} \bigg)^{\frac{1}{m}} \\
& \leq \prod_{k=1}^{m} \bigg( \sum_{j=1}^{N} \Big( \sum_{\boldsymbol{i} \in \mathcal{M}(m-1,N) } \Big\vert \card[\boldsymbol{i}]^{\frac{m-1}{2m}} m^{\frac{m-1}{2m}} \frac{a_{(\boldsymbol{i},_{k} j )}}{\sqrt{d(p_{(\boldsymbol{i},_{k} j )} ) }} \Big\vert^{2} \Big)^{1/2} \bigg)^{\frac{1}{m}} \\
& = m^{\frac{m-1}{2m}} \prod_{k=1}^{m} \bigg( \sum_{j=1}^{N} \Big( \sum_{\boldsymbol{i} \in \mathcal{M}(m-1,N) } \card[\boldsymbol{i}] \Big\vert  \frac{a_{(\boldsymbol{i},_{k} j )}}{\sqrt{d(p_{(\boldsymbol{i},_{k} j )} ) }} \Big\vert^{2} \Big)^{1/2} \bigg)^{\frac{1}{m}} \, .
\end{align*}
We now bound each one of the sums in the product. We use the fact that the coefficients $a_{\boldsymbol{j}}$ are symmetric. Also, if $q$ divides $p_{i_{1}} \cdots p_{i_{m}} = p_{\boldsymbol{i}}$, then it also divides
$p_{i_{1}} \cdots p_{i_{m}} p_{j} = p_{(\boldsymbol{i},_{k} j)}$; hence $d(p_{\boldsymbol{i}})  \leq d(p_{(\boldsymbol{i},_{k} j )} )$ for every $\boldsymbol{i}$ and every $j$. This altogether gives
\begin{multline*}
 \sum_{j=1}^{N} \Big( \sum_{\boldsymbol{i} \in \mathcal{M}(m-1,N) } \card[\boldsymbol{i}] \Big\vert  \frac{a_{(\boldsymbol{i},_{k} j )}}{\sqrt{d(p_{(\boldsymbol{i},_{k} j )} ) }} \Big\vert^{2} \Big)^{1/2} \\
= \sum_{j=1}^{N} \Big( \sum_{\boldsymbol{i} \in \mathcal{J}(m-1,N) } \card[\boldsymbol{i}]^{2} \frac{\vert  a_{(\boldsymbol{i},_{k} j )} \vert^{2} }{d(p_{(\boldsymbol{i},_{k} j )} )} \Big)^{1/2} \\
\leq \sum_{j=1}^{N} \Big( \sum_{\boldsymbol{i} \in \mathcal{J}(m-1,N) }  \frac{\vert \card[\boldsymbol{i}] a_{(\boldsymbol{i},_{k} j )} \vert^{2}}{d(p_{\boldsymbol{i}} )} \Big)^{1/2} \, .
\end{multline*}
Let us note that what we have here are the coefficients of an $(m-1)$-homogeneous polynomial in $N$ variables, we use now \eqref{Helson} to conclude our argument
\begin{align*}
 \sum_{j=1}^{N} \Big( \sum_{\boldsymbol{i} \in \mathcal{J}(m-1,N) } &  \frac{\vert \card[\boldsymbol{i}] a_{(\boldsymbol{i},_{k} j )} \vert^{2}}{d(p_{\boldsymbol{i})} )} \Big)^{1/2} \\
& \leq \sum_{j=1}^{N} \int_{\mathbb{T}^{N}} \Big\vert \sum_{\boldsymbol{i} \in \mathcal{J}(m-1,N) }   \card[\boldsymbol{i}] a_{(\boldsymbol{i},_{k} j )} w_{i_{1}} \cdots w_{i_{m-1}}  \Big\vert  d\sigma(w)\\
& \leq \int_{\mathbb{T}^{N}} \sum_{j=1}^{N} \Big\vert \sum_{\boldsymbol{i} \in \mathcal{M}(m-1,N) }   a_{(\boldsymbol{i},_{k} j )} w_{i_{1}} \cdots w_{i_{m-1}}  \Big\vert  d\sigma(w)\\
& \leq \sup_{z \in \mathbb{D}^{N}} \sum_{j=1}^{N} \Big\vert \sum_{\boldsymbol{i} \in \mathcal{M}(m-1,N) }   a_{(\boldsymbol{i},_{k} j )} z_{i_{1}} \cdots z_{i_{m-1}}  \Big\vert \\
& = \sup_{z \in \mathbb{D}^{N}}  \sup_{y \in \mathbb{D}^{N}} \Big\vert \sum_{j=1}^{N} \sum_{\boldsymbol{i} \in \mathcal{M}(m-1,N) }   a_{(\boldsymbol{i},_{k} j )} z_{i_{1}} \cdots z_{i_{m-1}} y_{j} \Big\vert \\
& \leq \Big( 1 - \frac{1}{m-1} \Big)^{m-1} \sup_{z \in \mathbb{D}^{N}} \Big\vert \sum_{\alpha \in \Lambda} c_{\alpha} z^{\alpha} \Big\vert \,,
\end{align*}
where the last inequality follows from a result of Harris \cite[Theorem~1]{Ha75} (see also \cite[(13)]{DeFrOrOuSe11}). This completes the proof.
\end{proof}

\noindent As we have already mentioned, very recently  \cite[Corollary 5.3]{BaPeSe13} has shown that for every $\varepsilon>0$ there exists $\kappa>0$ such that  \eqref{Bohnenblust Hille} holds with $\kappa (1+\varepsilon)^m$.
The main idea for the  proof is to replace \eqref{Blei} by a similar inequality in which we have mixed sums with $k$ and $m-k$ indices (instead of $1$ and $m-1$, as we have here). This allows to use instead of \eqref{Bayart} the following inequality:
\[
 \bigg( \sum_{\vert \alpha \vert = m} \vert c_{\alpha} \vert^{2}  \bigg)^{1/2} \leq c_{p}^{m} \bigg( \int_{\mathbb{T}^{N}} \Big\vert \sum_{\vert \alpha \vert = m} c_{\alpha} w^{\alpha} \Big\vert^{p} d\sigma (w) \bigg)^{\frac{1}{p}} \text{ for }  1 \leq p \leq 2 \,.
\]
A good control on the constants $c_{p}$ (that tend to $1$ as $p$ goes to $2$) gives the improvement on the constant in \eqref{Bohnenblust Hille} presented in \cite{BaPeSe13}. In our setting, by dividing by $\alpha +1$, we are using \eqref{Helson}, which already has constant
$1$. Hence this new approach does not improve the constants in our setting.\\

\noindent The authors wish to thank the referee for her/his careful reading and the suggestions and comments that helped to improve the paper.

\end{document}